\documentclass[12pt,leqno]{amsart}
\usepackage[dvips]{graphics}
\usepackage{amssymb}
\usepackage[breaklinks]{hyperref}
\hypersetup{
	colorlinks,
	linkcolor=black,
	citecolor=black
}

\numberwithin{equation}{section}

\newtheorem{thm}{Theorem}[section]

\newtheorem{lem}[thm]{Lemma}
\newtheorem{cor}[thm]{Corollary}
\newtheorem{prop}[thm]{Proposition}

 \theoremstyle{definition}

\theoremstyle{remark}
\newtheorem{rem}{Remark}[section]

\newcommand{\tref}[1]{Theorem~\ref{#1}}
\newcommand{\cref}[1]{Corollary~\ref{#1}}

\newcommand{\R}{\mathbb{R}}

\newcommand{\red}[1]{{\color{red}}}

\begin{document}
\pagebreak


\title{Flat subsets of Euclidean buildings}

\author{Raphael Appenzeller,  Auguste H\'ebert, Alexander Lytchak}


\date{\today}  

\begin{abstract}
We prove that any convex flat subset in a complete Euclidean building is contained in an apartment of the maximal system of apartments.
\end{abstract}

\maketitle
\renewcommand{\theequation}{\arabic{section}.\arabic{equation}}
\pagenumbering{arabic}



\section{Introduction}
Euclidean buildings were introduced by  Jacques Tits in \cite{Tits} and play an important role in different areas of mathematics.  Their metric structure played a central role in super-rigidity results starting with  \cite{KleinerLeeb1}, \cite{GS}.  We refer to \cite{Linus}, \cite{KLM} for further investigations of their metric properties and to  \cite{BM} and the literature therein for super-rigidity questions related to Euclidean buildings.

The following result, possibly known to Bruce Kleiner and Bernhard Leeb, and  verified  in \cite{KleinerLeeb1} in a special case, has been discussed in other special cases in several independent works. In  this note we provide a proof of this result:

\begin{thm} \label{thm-main}
	Let $X$ be a Euclidean  building.  Let $C\subset X$ be a convex subset of $X$
	which is isometric to a convex subset of some Euclidean space.  Then $C$ is contained in an apartment $A$  of $X$.

	Moreover, if for some $x\in C$ we choose an apartment  $A^x$  in the infinitesimal Euclidean building  $T_xX$ with $T_xC\subset A^x$ then $A$ can be chosen so that $T _xA=A^x$. 
\end{thm}

 	All Euclidean buildings in this note, in particular, in Theorem \ref{thm-main}, are assumed to be complete, cf. Remark \ref{rem: non-complete} below for a comment about possible generalizations.

The analogue of the first part of this result  for spherical buildings is contained in   \cite{KleinerLeeb1}, see Theorem \ref{thm-klsph} below. 
In some special cases,  Theorem \ref{thm-main} is contained in \cite{KleinerLeeb1} as well.

A different proof of the spherical theorem  was obtained later in \cite[Proposition 1.3]{BL} in a slightly more general setting.  Moreover, this approach also covered the case of Theorem \ref{thm-main} for discrete Euclidean buildings, \cite[Remark 1.4]{BL}. 

For some convex subsets in discrete Euclidean buildings,  Theorem \ref{thm-main} is presented in \cite[Theorem 11.53]{Abram}. The case of convex triangles $C$ in Bruhat--Tits buildings appears in \cite[Exercise 2.3.10.3]{Rousseau}.  Several  special cases of  Theorem \ref{thm-main} have been obtained and used in \cite{BM}.

 \begin{rem} \label{rem: non-complete}
	Here is the only place in this paper, where we discuss non-complete buildings, we refer to  \cite{Parreau}, \cite{Str} for their discussion.

	In any non-complete Euclidean building, Theorem \ref{thm-main} cannot hold as stated. Indeed, any such building $X$ contains a (half-open) geodesic $\gamma:[0,1) \to X$, whose image $C$ is closed as a subset of $X$.  Such $C$ cannot be contained in an apartment. 
	However, under the additional assumption that $C$ is itself complete, we hope to discuss conditions for when $C$ is contained in an apartment in a forthcoming work.  
\end{rem}

  In Section \ref{sec: sph} we follow \cite{KleinerLeeb1}, in order to  reprove and  to extend the spherical version of our main theorem.   One major difference between the spherical and the Euclidean case is  the  finiteness of the simplicial decomposition of spherical apartments, which allows for a choice of a potential enlargement $A$  of the convex subset $C$.  In Section 4,  we prove Theorem \ref{thm-main} by enlarging the convex set step by step.  In order to make this enlargement procedure  successful, we deal with the case of convex polyhedra $C$ first, 
  and then deduce the result for  general convex subsets  by an appropriate approximation.

{\bf Acknowledgments.} Alexander Lytchak is grateful to the authors of \cite{BM} for their interest in this topic, especially  to Christine Breiner, for several related questions, which gave rise to this project.    Raphael Appenzeller would like to thank Isobel Davies for discussions. Raphael Appenzeller and Auguste H\'ebert acknowledge support from the Procope project
``Buildings, galleries and beyond'' and would like to thank its members, St\'ephane Gaussent, Zijun Li, Bianca Marchionna, Paul Philippe and Petra Schwer for helpful discussions. 
Raphael Appenzeller was partially supported by RTG 2229 ``Asymptotic Invariants and Limits of Groups and Spaces'' funded by Deutsche Forschungsgemeinschaft (DFG, German Research Foundation).

\section{Notation}
\label{subsec-notations}
We assume the reader to be familiar with the theory of Euclidean buildings and  CAT($\kappa$) spaces.
We are going to borrow the terminology  from  \cite{KleinerLeeb1}.   We consider spherical buildings with their canonical CAT(1) metric and Euclidean buildings with their canonical CAT(0) metric.  We assume that all buildings are complete.

Spherical and Euclidean  buildings involved are not necessarily thick nor irreducible. However, one can always consider a canonical reduction  as a spherical join of thick spherical buildings and $\mathbb S^0$'s, respectively, as a direct product of thick Euclidean buildings and $\R$'s, \cite[Sections 3.7, 4.9]{KleinerLeeb1}.  We always consider a maximal system of apartments, thus an apartment in an $n$-dimensional Euclidean  building is just an $n$-dimensional convex subset isometric to $\R^n$, \cite[Corollary 4.6.2]{KleinerLeeb1}.

For a Euclidean building $X$ we denote by $\partial _{\infty } X$ its spherical building at infinity.  For $x\in X$ we denote by $\Sigma _x X$ the space of directions at $x$, an $(n-1)$-dimensional spherical building. By $T_xX$ we denote the tangent space of $X$ at $x$, thus the Euclidean cone 
$\text{Cone}(\Sigma _x X)$.

In a spherical building  $B$ we denote a subset $C$ as convex if for any pair of points $x,y\in C$ with distance less than $\pi$, the set $C$ contains the unique geodesic in $B$ between $x$ and $y$.

For a Euclidean building $X$ and a point $x$ we have two canonical surjective $1$-Lipschitz maps, the \emph{logarithmic maps}  (as in any geodesically complete CAT(0) space)
 \cite[p. 124-128]{KleinerLeeb1},

$$\log_x:\partial _{\infty} X \to \Sigma _xX    \; \;  \text{and} \; \; \text{Log} _x :X\to T_xX\,.$$

\section{Spherical observations} \label{sec: sph}
For a Euclidean building $X$ the logarithmic map $\log_x :\partial _{\infty} X \to \Sigma _x X$ is surjective and $1$-Lipschitz  \cite[p. 128]{KleinerLeeb1}.
This, applied to the cone $X=\operatorname{Cone}(B)$ over a spherical building $B$, provides for any point $x\in B$ a surjective $1$-Lipschitz map $\tilde {\log} _ x$  from $B=\partial_{\infty} X$ to the spherical suspension  $\mathbb S^0 *\Sigma _xB =\Sigma _xX$ (cf. \cite[Section 2.6]{Lbuild}),

In a more general context such surjective $1$-Lipschitz maps are investigated 
in \cite{Lbuild}; here we only need the following observation relating such maps to  spreading morphisms
used in \cite{KleinerLeeb1}.

\begin{lem} \label{lem: slm}
A surjective $1$-Lipschitz map $f:B_1\to B_2$ between spherical buildings 
restricts to an isometry on each chamber

$\Delta \subset B_1$.   

\end{lem}

\begin{proof}
Fix any $x_1,x_2\in \Delta$.  Consider $y_1=f(x_1)$, an arbitrary point $\bar y_1\in B_2$ with $d(y_1,\bar y_1)=\pi$ 
and a point $\bar x_1 \in B_1$ in the preimage of $\bar y_1$.  

The diameter of $B_1$ is $\pi$ and  $f$ is $1$-Lipschitz. Therefore,
$d(x_1,\bar x_1)=\pi$ and $f$ restricts to an isometry on any geodesic from $x_1$ to $\bar x_1$.

Since $x_1,x_2 \in \Delta$, there exists a geodesic from $x_1$ to $\bar x_1$ passing through $x_2$.  Thus, $f$ preserves the distance between $x_1$ and $x_2$.
\end{proof}

If $\dim (B_1)=\dim (B_2)$, a  map $f$ satisfying the  conclusion of the above Lemma is called in \cite[Section 3.11]{KleinerLeeb1} a \emph{spreading morphism}.

Let $f:B_1\to B_2$ be a surjective $1$-Lipschitz map between spherical $n$-dimensional buildings.  Since the restriction of $f$ to each chamber is an isometry, for any $x\in B_1$, any sufficiently short geodesic starting at $x$ is mapped isometrically onto its image starting at $y:=f(x)$.  This defines a differential $\Sigma _x f:\Sigma _xB_1 \to \Sigma _yB_2$.  This  differential is $1$-Lipschitz by definition and it is surjective due to \cite[Lemma 3.11.1]{KleinerLeeb1}.

We will call a metric space $C$ \emph{spherical} if it isometric to a subset of the unit sphere $\mathbb S^k$ of some dimension $k$.

The  following result is \cite[Proposition 3.9.1]{KleinerLeeb1}.

\begin{thm}[Kleiner--Leeb] \label{thm-klsph}
Let $B$ be a spherical building, let $C\subset B$ be a closed, convex, spherical subset. Then there exists an apartment $A\subset B$ which contains $C$.
\end{thm}

 Using essentially the same proof, we obtain the following extension:

	\begin{prop} \label{prop: main}
		Let $f:B_1\to B_2$ be a surjective $1$-Lipschitz map between $n$-dimensional spherical buildings.  Let $C_1$ be a convex spherical 
		subset of $B_1$ which is mapped by $f$ isometrically onto its image $C_2 \subset B_2$.
		
		Assume that $A_2$ is an apartment in $B_2$ containing $C_2$.  Then there exists an apartment $A_1$ in $B_1$ containing $C_1$ such that 
		$f$ maps $A_1$ isometrically onto $A_2$.		
	\end{prop}

	\begin{proof}
	We use induction on $n$ as in   \cite[Proposition 3.9.1]{KleinerLeeb1}.

	 If $n=0$ then $A_2$ consists of two points $y, \bar y$ and $C_1$ is either an apartment,  a singleton  or the empty set. 	 The first  case is clear.  	If $C_1=\emptyset$, we  choose arbitrary preimages $p\in f^{-1} (y)$ and $\bar p  \in  f^{-1} (\bar y)$.  Then $\{p, \bar p\}$ is the required apartment $A_1$.  Similarly, if $C_1$ is a singleton $p\in f^{-1} (y)$, we choose an arbitrary $\bar p \in f^{-1} (\bar y)$. Then, again, $A_1:=\{p, \bar p\}$ is the required apartment.

	Assume now the result is true in dimension $n-1$.    Assume first that 
	$C_1$ is a pair of antipodal points $\{p, \bar p \}$.  Set $y=f(p)$ and $\bar y=f(\bar p)$.  We apply  the inductive assumption to the surjective $1$-Lipschitz differential $\Sigma_p f:\Sigma _p B_1\to \Sigma _y B_2$. Hence, we obtain an apartment $A^p\subset \Sigma _p B_1$, mapped by $\Sigma _p f$ isometrically to $\Sigma _y A_2$.  
	
	There is a unique  apartment $A$ in $B_1$, which contains $p,\bar p$  and has $A^p$
	as its space of directions $\Sigma _p A= A^p$.  Since the map $f$ restricts to an isometry on any geodesic between $p$ and $\bar p$ and since $\Sigma _p f$ restricts to an isometry from $\Sigma _p A$ to $\Sigma _y A_2$, we deduce that $f$ sends $A$ isometrically onto $A_2$, finishing the proof in the case $C_1=\{p, \bar p\}$.
	
	If $C_1=\emptyset$, we choose arbitrary antipodes $y, \bar y \in A_2$ and arbitrary preimages $p\in f^{-1} (y)$ and $\bar p\in f^{-1} (\bar y)$. We then apply the previously discussed case and find an apartment $A_1$ containing $p,\bar p$ and sent by $f$ isometrically onto $A_2$

Similarly, if  $C_1=\{p\}$ is a singleton, we set $y=f(p)\in A_2$ and consider the antipode $\bar y$ of $y$ in $A_2$. We then choose an arbitrary $\bar p \in f^{-1} (\bar y)$ and apply  the previously discussed case to find the required apartment $A_1$ containing $p$ and $\bar p$.

	Thus, we may assume that $C_1$ has more than two points. In particular, $C_1$ is connected in this case.

	Consider the set $\mathcal A$ of all apartments in $B_1$ which are mapped by $f$ isometrically onto $A_2$.   Applying the previously discussed case, we find some $A\in  \mathcal A$, which contains an arbitrarily fixed point $x\in C_1$.

	Among all $A\in \mathcal A$ consider an apartment  $A_1$ such that the number of its open faces (of all possible dimensions) 
	 intersecting $C_1$ is as large as possible in $\mathcal A$.

We claim that $C_1\subset A_1$.  Assuming otherwise, we find a point $p\in C_1 \cap A_1$ and some $q\in C_1\setminus A_1$ contained in the open star of $p$.

Consider the opposite point $\bar p$ of $p$ in $A_1$ and the union $V\subset A_1$ of all geodesics in $A_1$ between $p$ and $\bar p$, which contain at least one  point in $C_1 \setminus \{p, \bar p\}$.  Then $V$   is  a convex  subset of $A_1$ which contains $A_1\cap C_1$.

Set $y=f(p)$.
By the inductive assumption applied to the differential $\Sigma _p f:\Sigma _p B_1  \to \Sigma _y B_2$ we find an apartment 
$K\subset \Sigma _p B_1$, which is mapped isometrically by $\Sigma _p f$  onto $\Sigma _{y} A_2$ and contains $\Sigma _p C_1$.  Then we define the  apartment $A'\subset B_1$ to consist of all geodesics between $p$ and $\bar p$ starting in directions of $K$.

Apartment $A'$ contains $V$ but also a neighborhood of $p$ in $C_1$, hence the point $q$.    The map $f$ sends each geodesic between $p$ and $\bar p$ isometrically onto its image in $A_2$. Since the differential $\Sigma_pf:K\to \Sigma _y A_2$ is  an isometry, the restriction of $f$ to $A'$ is an isometry. The apartment $A'$ contains $V$ hence all of $A_1\cap C_1$ but also the point $q$.
  Thus, the open face of $A'$ containing 
$q$ does not intersect  $A_1$. This contradicts the maximality of $A_1$ and proves that $A_1$ contains all of $C_1$.	
	\end{proof}

As direct consequences we deduce a minor strengthening of Theorem \ref{thm-klsph}, analogous to the second part of our main Theorem \ref{thm-main}.

\begin{cor}
Let $C\subset B$ be a closed, convex, spherical subset in a spherical building $B$. Let $x\in C$ be arbitrary and let $A_x\subset \Sigma _x B$ be an apartment in the space of directions such that $\Sigma _xC\subset A_x$.  Then there exists an apartment $A\subset B$ which contains $C$ and such that $\Sigma _xA=A_x$.
\end{cor}
 
\begin{proof}
We apply Proposition \ref{prop: main} to  the canonical surjective $1$-Lipschitz map $\log_x :B\to  S^{0} * \Sigma _x B$,  with $C_1=C$ and $A_2=A_x$.
\end{proof}

Another direct consequence is the statement of Theorem \ref{thm-main} for a special class of convex subsets.

\begin{cor} \label{cor: cone}
Let $C$ be a flat, closed, convex subset in a Euclidean building $X$. Assume that $C$ is a Euclidean cone with its tip at the point $x$.  Let $A_x$ be an apartment in the tangent cone $T_xX$, which contains $T_xC$.  Then $C$ is contained in an apartment $A$, such that $T_xA=A_x$.
\end{cor}

\begin{proof}
We apply Proposition \ref{prop: main} to the logarithmic map $\log_x: \partial _{\infty} X\to \Sigma _xX$  with $C_1=\partial _{\infty} C$ and $A_2$ being the unit sphere in $A_x$.  Hence, we find  an apartment $A_1$ in $\partial _{\infty} X$, which contains $C_1$ and  is mapped isometrically by $\log _x$ onto the unit sphere $A_2\subset A_x$.

This apartment  $A_1$ is the boundary at infinity of an apartment $A$ in $X$ which contains $x$  \cite[Lemma 2.3.4]{KleinerLeeb1}. Since $\partial_{\infty} A_1$ contains $C_1$, we deduce that $A$ contains $C$.
\end{proof}
	
As a special case, Corollary \ref{cor: cone} covers the case of convex subsets $F$ isometric to $\R^k$.  In this case, $x$ can be chosen arbitrarily  and the union of all apartments $A_x\subset T_xX$ containing $T_xF$ is exactly the parallel set $P(T_xF)$ of $T_xF$ in $T_xX$
(see \cite[Section 4.8]{KleinerLeeb1}).  Moreover, the union of all apartments $A$ containing $F$ is the parallel set $P(F)$ of $F$.  Thus, Corollary \ref{cor: cone} implies the following result (cf. \cite{BM}):

\begin{cor}
Let $F\subset X$ be a convex subset isometric to a Euclidean space in a Euclidean building $X$. Then for every $x\in F$ the tangent space $T_x(P(F))$ at $x$ to the parallel set $P(F)$ of $F$ coincides with the parallel set $P(T_xF)$ of the flat $T_xF$ in the building $T_xX$. 
\end{cor}

 \section{Main Theorem}

	Before embarking on the proof of the main theorem, we recall that a \emph{convex polyhedron} is an intersection of finitely many half-spaces in a Euclidean space.  Any $k$-dimensional convex polyhedron $\mathcal P$, is uniquely (up to isometries of the ambient Euclidean space) embedded into a $k$-dimensional Euclidean space $V$, its \emph{affine hull}.  
	 By \emph{facets} we denote faces of $\mathcal P$ of codimension $1$. Each facet $\mathcal F\subset \mathcal P\subset V$ defines a unique  halfspace $H_\mathcal F$ of $V$ containing $\mathcal P$, such that $\mathcal F\subset \partial H_{\mathcal F}$.

	A polyhedron $\mathcal P$ has no facets if and only if $\mathcal P$ is itself a Euclidean space $\mathcal P=V$.   If $\mathcal P$ has $l\geq 1$ facets and if $\mathcal F$ is a facet of $\mathcal P$ we consider the  polyhedron $\mathcal P_\mathcal F$  (\emph{extension of the polyhedron $\mathcal P$  beyond its facet $\mathcal F$}) given by the intersection of $l-1$ halfspaces 
	$$\mathcal P_\mathcal F:=\bigcap_{\mathcal G} H_\mathcal G\,,$$
	where $\mathcal G$ runs over facets of $\mathcal P$ different from $\mathcal F$.  By construction, $\mathcal P_\mathcal F$  contains $\mathcal P$ and has $l-1$ facets.

	A subset of a Euclidean  building that is isometric to a convex polyhedron will also be called a \emph{convex polyhedron}.  Such a subset is automatically a convex subset of the Euclidean building.

We will use the following observation in CAT(0) geometry (only in the case of Euclidean buildings):

\begin{lem} \label{lem: final}
	Let $X$ be a CAT(0) space.  Let $Y_1,Y_2 $ be closed, convex subsets of $X$ isometric to convex subsets of some Euclidean spaces.  Let $x\in  Y_1\cap Y_2$ be a point and assume that
	\begin{enumerate} 
	\item $T:={\operatorname{Log}}_x (Y_1) \cup {\operatorname{Log}}_x (Y_2) \subset T_xX \,$ is a convex subset of $T_xX$,
	\item ${\operatorname{Log}}_x (Y_1) \cap {\operatorname{Log}}_x (Y_2) ={\operatorname{Log}}_x (Y_1\cap Y_2)\,.$
	\end{enumerate}
	Then $Y:=Y_1\cup Y_2$ is a  convex subset of $X$ isometric to $T$. 
\end{lem}

	\begin{proof}
		By (2), the logarithmic map restricts to a bijective map
	$$I:=\operatorname{Log}_x :Y\to T\,.$$

Since $Y_1$ and $Y_2$ are flat and convex, $I$ restricts to isometries
 $$I: Y_i \to I(Y_i)\subset T_x (Y_i)\subset T\subset T_xX\,.$$

		Choose arbitrary $y_1,y_2\in Y$ and consider the unique geodesic  $\gamma$ between $I(y_1)$ and $I(y_2)$ in $T$, which exists by (1).  Then $\gamma$ is a concatenation of at most one geodesic in $I(Y_1)$ and at most one geodesic in $I(Y_2)$, hence  the curve $I^{-1}(\gamma)$ has the same length as $\gamma$.  Since $I$ is $1$-Lipschitz, 
		 the curve $I^{-1} (\gamma) \subset Y_1\cup Y_2$  has length $d(y_1,y_2)$.

		 Thus $Y_1\cup Y_2$ is convex and $I:Y\to T$ is an isometry.
	\end{proof}

\begin{proof}[Proof of Theorem \ref{thm-main}]

We fix the $n$-dimensional Euclidean building 
and proceed by induction on the dimension $k$ of the convex subset $C$. 

   If $k=0$ then 
$C$ is a single point $x$ and the result is well-known in this case,  \cite[Lemma 4.2.3]{KleinerLeeb1}.

Assume  the statement of Theorem \ref{thm-main}  to be  true for all flat, convex subsets of dimension less than $k$. 

{\bf Step 1.}  We are going to prove that any flat, convex $k$-dimensional polyhedron $\mathcal P \subset X$ is contained in some apartment.  We proceed by induction  on the number of facets of $\mathcal P$.  If $\mathcal P$ has no facets then $\mathcal P$ is isometric to a Euclidean space $\R^k$ and the result is contained in Corollary \ref{cor: cone} or in \cite[Proposition 4.6.1]{KleinerLeeb1}.

Assume that all convex, $k$-dimensional  polyhedra in $X$ with less than $l \geq 1$ facets  are contained in apartments and let a convex polyhedron $\mathcal P$ with $l$ facets  be fixed.    Choose a facet $\mathcal F$ of 
$\mathcal P$.

Choose an arbitrary interior point $x$ of the facet $\mathcal F$.  Then $T_x \mathcal P$ is a $k$-dimensional, flat halfspace  in the tangent space $T_xX$.  By  Corollary \ref{cor: cone}, we find an apartment $A^x$ in the (infinitesimal) Euclidean building $T_xX$ with $T_x \mathcal P \subset A^x$.

We apply the inductive assumption on the $(k-1)$-dimensional, convex polyhedron $\mathcal F$ and find an apartment $A$ of $X$ with $\mathcal F\subset A$ and  $ T_xA =A^x$.

Since $\mathcal P$ is convex and flat, $\operatorname{Log}_x$ restrict to an isometry 
$$\operatorname{Log}_x :\mathcal P \to \operatorname{Log}_x (\mathcal P)=: \mathcal P ' \subset T_xX\,.$$ 
 
 Consider the facet $\mathcal F':= \operatorname{Log}_x (\mathcal F)$ of $\mathcal P'$
 and  the polyhedron $\mathcal P'_{\mathcal F'}$, the   extension of the polyhedron $\mathcal P'$ beyond its face $\mathcal F'$ within the affine hull $V'$ of $\mathcal P' \subset A^x$, as we have defined it at the beginning of this section.  This is a $k$-dimensional convex polyhedron with less than $l$ facets, contained in the Euclidean space $A^x$.
 
 The polyhedron $\mathcal P' _{\mathcal F'}$ is the union of two subpolyhedra:
 $\mathcal P'$ and the closure $\mathcal C'$ of its complement
  $$\mathcal C':=\overline {\mathcal P' _{\mathcal F'} \setminus \mathcal P'}\,,$$
 which intersect in $\mathcal F'$.
 
 We consider now the polyhedron $\mathcal C \subset A$, uniquely defined by the condition $\operatorname{Log}_x (\mathcal C)= \mathcal C'$  (note that 
 $\operatorname{Log}_x :A\to A^x$ is an isometry).   We define $\hat {\mathcal P} $ as
 $$\hat {\mathcal P} = \mathcal P \cup \mathcal C\;.$$
 Then $\hat {\mathcal P} \subset X$ is a union of two flat, convex subsets.  The map $\operatorname{Log}_x$ sends $\mathcal P$ isometrically to $\mathcal P'$ and $\mathcal C$ to $\mathcal C'$. We can now apply Lemma \ref{lem: final} to conclude that $\hat {\mathcal P}$ is a convex subset of $X$ isometric to the polyhedron $\mathcal P'_{\mathcal F'}$.  Since this polyhedron has less than $l$ facets, by induction this polyhedron is contained in some apartment of $X$. Since $\mathcal P \subset \hat {\mathcal P}$, this finishes the proof of {\bf Step 1}.

{\bf Step 2.}  We are going to show that any flat, convex, $k$-dimensional  subset $C$ of $X$ is contained in an apartment (without prescribing the infinitesimal behavior at a point).
We choose an arbitrary increasing sequence of convex  polyhedra $C_1\subset C_2 \subset ...$ contained  in $C$,  such that $C$ is contained in the closure of the union $\bigcup C_i$.

Due to Step 1, we find for each $C_i$ an apartment $A_i$ in $X$ containing  $C_i$.  For each $i$, consider the convex subset 
$$\hat C_i:= \bigcap _{j\geq i}  A_j\;.$$
Then $\hat C_i$ is a convex subset of  the apartment $A_i$
and $C_i \subset \hat C_i$.  Moreover,  $(\hat C_i)_i$ constitutes an increasing sequence of flat, convex subsets of $X$.  

Any  $\hat C_i$ is a Weyl polyhedron as an intersection of apartments, \cite[Corollary 4.4.6]{KleinerLeeb1}. Thus,  it is a polyhedron and the number of facets of $\hat C_i$  is bounded only in terms of the Weyl group $W$ of the building $X$. Moreover, $$\hat C:= \overline{\bigcup _{i} \hat C_i}$$ is complete and thus its image by any isometry is complete and hence closed (here is the only place where we use the completeness of $X$ in an essential way!). Therefore $\hat{C}$ is a convex polyhedron.   Thus,  applying Step 1 again, we deduce that 
$\hat C$ is contained in some apartment of $X$. Since $C\subset \hat C$, this finishes the proof of Step 2.

{\bf Step 3.}   We are now going to prove the remaining "Moreover part" of Theorem \ref{thm-main} for  $k=\dim (C)$.  Note, that we have already used this "Moreover part"  for smaller dimensions in Step 1 above.

 Let   $C\subset X$ be a closed, convex, flat $k$-dimensional subset of $X$, let $x\in C$ be arbitrary and let $A_x$ be an apartment in $T_xX$ containing the convex subset $T_xC$.   

By Step 2, we find some apartment $A'$ in $X$ containing $C$.  Consider
the cone $C' \subset A'$  with tip at $x$ and  $T_xC =T_xC'$.   Then $C\subset C'$.  We  
 apply  Corollary \ref{cor: cone} to the cone $C'$  and find an apartment $A$ in $X$  with $C'\subset A$ and $T_xA= A_x$.   Since $C\subset C'$, we have $C\subset A$.

This finishes the proof of the induction step and  of the theorem.
\end{proof}

\bibliographystyle{alpha}
\bibliography{Buildings}

@misc{BM,
      title={Harmonic Maps into Euclidean Buildings and Non-Archimedean Superrigidity}, 
      author={Breiner, C. and Dees, B. K. and Mese, C.},
      year={2024},
      eprint={2408.02783},
      archivePrefix={arXiv},
      primaryClass={math.DG},
      url={https://arxiv.org/abs/2408.02783}, 
}

@article {Str,
    AUTHOR = {Struyve, K.},
     TITLE = {({N}on)-completeness of {$\Bbb R$}-buildings and fixed point
              theorems},
   JOURNAL = {Groups Geom. Dyn.},
  FJOURNAL = {Groups, Geometry, and Dynamics},
    VOLUME = {5},
      YEAR = {2011},
    NUMBER = {1},
     PAGES = {177--188},
      ISSN = {1661-7207,1661-7215},
   MRCLASS = {51E24 (20E42 20F65)},
  MRNUMBER = {2763784},
MRREVIEWER = {Jean\ L\'ecureux},
       DOI = {10.4171/GGD/121},
       URL = {https://doi.org/10.4171/GGD/121},
}

@incollection {Parreau,
    AUTHOR = {Parreau, A.},
     TITLE = {Immeubles affines: construction par les normes et \'etude des
              isom\'etries},
 BOOKTITLE = {Crystallographic groups and their generalizations ({K}ortrijk,
              1999)},
    SERIES = {Contemp. Math.},
    VOLUME = {262},
     PAGES = {263--302},
 PUBLISHER = {Amer. Math. Soc., Providence, RI},
      YEAR = {2000},
      ISBN = {0-8218-2001-X},
   MRCLASS = {20E42 (20G25 51E24 53C23)},
  MRNUMBER = {1796138},
MRREVIEWER = {Guy\ Rousseau},
       DOI = {10.1090/conm/262/04180},
       URL = {https://doi.org/10.1090/conm/262/04180},
}

@article {BL,
    AUTHOR = {Balser, A. and Lytchak, A.},
     TITLE = {Building-like spaces},
   JOURNAL = {J. Math. Kyoto Univ.},
  FJOURNAL = {Journal of Mathematics of Kyoto University},
    VOLUME = {46},
      YEAR = {2006},
    NUMBER = {4},
     PAGES = {789--804},
      ISSN = {0023-608X},
   MRCLASS = {53C23},
  MRNUMBER = {2320351},
MRREVIEWER = {Koichi Nagano},
       DOI = {10.1215/kjm/1250281604},
       URL = {https://doi.org/10.1215/kjm/1250281604},
}

@article {KleinerLeeb1,
    AUTHOR = {Kleiner, B. and Leeb, B.},
     TITLE = {Rigidity of quasi-isometries for symmetric spaces and
              {E}uclidean buildings},
   JOURNAL = {Inst. Hautes \'{E}tudes Sci. Publ. Math.},
  FJOURNAL = {Institut des Hautes \'{E}tudes Scientifiques. Publications
              Math\'{e}matiques},
    VOLUME = {86},
      YEAR = {1997},
     PAGES = {115--197 (1998)},
      ISSN = {0073-8301},
   MRCLASS = {53C35 (20E42 20F32 53C23 57M07)},
  MRNUMBER = {1608566},
MRREVIEWER = {Lee Mosher},
       URL = {http://www.numdam.org/item?id=PMIHES_1997__86__115_0},
}

@book {Abram,
    AUTHOR = {Abramenko, P. and Brown, K.},
     TITLE = {Buildings},
    SERIES = {Graduate Texts in Mathematics},
    VOLUME = {248},
      NOTE = {Theory and applications},
 PUBLISHER = {Springer, New York},
      YEAR = {2008},
     PAGES = {xxii+747},
      ISBN = {978-0-387-78834-0},
   MRCLASS = {20E42 (20F55 20J06 51E24 51F15)},
  MRNUMBER = {2439729},
MRREVIEWER = {Ralf\ Koehl},
       DOI = {10.1007/978-0-387-78835-7},
       URL = {https://doi.org/10.1007/978-0-387-78835-7},
}

@article {Lbuild,
    AUTHOR = {Lytchak, A.},
     TITLE = {Rigidity of spherical buildings and joins},
   JOURNAL = {Geom. Funct. Anal.},
  FJOURNAL = {Geometric and Functional Analysis},
    VOLUME = {15},
      YEAR = {2005},
    NUMBER = {3},
     PAGES = {720--752},
      ISSN = {1016-443X},
     CODEN = {GFANFB},
   MRCLASS = {53C24},
  MRNUMBER = {2221148},
MRREVIEWER = {Koichi Nagano},
       DOI = {10.1007/s00039-005-0519-6},
       URL = {http://dx.doi.org/10.1007/s00039-005-0519-6},
}

@article {KLM,
    AUTHOR = {Kapovich, M. and Leeb, B. and Millson, J.},
     TITLE = {Polygons in buildings and their refined side lengths},
   JOURNAL = {Geom. Funct. Anal.},
  FJOURNAL = {Geometric and Functional Analysis},
    VOLUME = {19},
      YEAR = {2009},
    NUMBER = {4},
     PAGES = {1081--1100},
      ISSN = {1016-443X,1420-8970},
   MRCLASS = {53C21 (20E42 51E24 53C20)},
  MRNUMBER = {2570316},
MRREVIEWER = {Athanase\ Papadopoulos},
       DOI = {10.1007/s00039-009-0026-2},
       URL = {https://doi.org/10.1007/s00039-009-0026-2},
}

@Book{Rousseau,
  Author = {Rousseau, G.},
  Title = {Euclidean buildings. {Geometry} and group actions},
  FSeries = {EMS Tracts in Mathematics},
  Series = {EMS Tracts Math.},
  Volume = {35},
  ISBN = {978-3-98547-039-6; 978-3-98547-539-1},
  Year = {2023},
  Publisher = {Berlin: European Mathematical Society (EMS)},
  Language = {English},
  DOI = {10.4171/ETM/35},
  zbMATH = {7742947},
  Zbl = {1533.51001}
}

@incollection{Linus,
 author = {Kramer, L.},
 title = {Metric properties of {Euclidean} buildings},
 booktitle = {Global differential geometry},
 isbn = {978-3-642-22841-4; 978-3-642-22842-1},
 pages = {147--159},
 year = {2012},
 publisher = {Berlin: Springer},
 language = {English},
 doi = {10.1007/978-3-642-22842-1_6},
 zbMATH = {6054002},
 Zbl = {1270.51014}
}

@incollection {Tits,
    AUTHOR = {Tits, J.},
     TITLE = {Immeubles de type affine},
 BOOKTITLE = {Buildings and the geometry of diagrams ({C}omo, 1984)},
    SERIES = {Lecture Notes in Math.},
    VOLUME = {1181},
     PAGES = {159--190},
 PUBLISHER = {Springer, Berlin},
      YEAR = {1986},
      ISBN = {3-540-16466-9},
   MRCLASS = {20G15 (51B25)},
  MRNUMBER = {843391},
MRREVIEWER = {W.\ M.\ Kantor},
       DOI = {10.1007/BFb0075514},
       URL = {https://doi.org/10.1007/BFb0075514},
}

@article {GS,
    AUTHOR = {Gromov, M. and Schoen, R.},
     TITLE = {Harmonic maps into singular spaces and {$p$}-adic
              superrigidity for lattices in groups of rank one},
   JOURNAL = {Inst. Hautes \'Etudes Sci. Publ. Math.},
  FJOURNAL = {Institut des Hautes \'Etudes Scientifiques. Publications
              Math\'ematiques},
    NUMBER = {76},
      YEAR = {1992},
     PAGES = {165--246},
      ISSN = {0073-8301,1618-1913},
   MRCLASS = {58E20 (22E40)},
  MRNUMBER = {1215595},
MRREVIEWER = {Caio\ J. C. Negreiros},
       URL = {http://www.numdam.org/item?id=PMIHES_1992__76__165_0},
}

\end{document}